\documentclass{amsart}

\usepackage{amsmath,amssymb,amsthm}
\usepackage{graphicx,tikz}
\usetikzlibrary{shapes,arrows,positioning,calc}

\newtheorem*{thm}{Theorem}

\newtheorem{lemma}{Lemma}

\theoremstyle{definition}

\theoremstyle{remark}

\begin{document}

\title[]{Eigenvector Phase Retrieval: recovering Eigenvectors from the absolute value\\ of their entries}
\thanks{  S.S. was partly supported by the NSF (DMS-2123224) and the Alfred P. Sloan Foundation.}

\author[]{Stefan Steinerberger}
\address{Department of Mathematics, University of Washington, Seattle, WA 98195, USA}
\email{steinerb@uw.edu}

\author[]{Hau-tieng Wu}
\address{Department of Mathematics and Department of Statistical Science, Duke University, Durham, NC, USA}
\email{hauwu@math.duke.edu}

\begin{abstract}
We consider the eigenvalue problem $Ax = \lambda x$ where $A \in \mathbb{R}^{n \times n}$ and the eigenvalue is also real $\lambda \in \mathbb{R}$. If we are given $A$, $\lambda$ and, additionally, the absolute value of the entries of $x$ (the vector $(|x_i|)_{i=1}^n$), is there a fast way to recover $x$? In particular, can this be done quicker than computing $x$ from scratch? This may be understood as a special case of the phase retrieval problem. We present a randomized algorithm which provably converges in expectation whenever $\lambda$ is a simple eigenvalue. The problem should become easier when $|\lambda|$ is large and we discuss another algorithm for that case as well.
\end{abstract}

\maketitle

\section{Introduction  }
We discuss the following problem: given a matrix $A \in \mathbb{R}^{n \times n}$ having a real eigenvalue $\lambda \in \mathbb{R}$ with eigenvector $Ax = \lambda x$, is it possible to `quickly' recover the eigenvector $x \in \mathbb{R}^n$ from knowing the absolute values of all its entries, i.e. $|x_i|_{i=1}^n$? We assume throughout the paper that both $A \in \mathbb{R}^{n \times n}$ and $\lambda \in \mathbb{R}$ are known and that the eigenvalue $\lambda$ is simple which makes the associated eigenvector $x$ unique (up to sign): the problem is to recover the sign vector $\varepsilon = (\varepsilon_i)_{i=1}^{n}$ defined via the equation $x_i = \varepsilon_i |x_i|$. Naturally, given $A$ and $\lambda$, we can find the eigenvector $x$ by look for a vector in the nullspace
$$0 \neq x \in \mbox{ker} (A - \lambda \cdot \mbox{Id}_{n \times n}).$$
Can one do better by using knowledge about the size of the entries?
\begin{quote}
\textbf{Problem.} When computing an eigenvector $Ax = \lambda x$, does it help to know the magnitude $|x_i|_{i=1}^{n}$ of all $n$ entries of the eigenvector in advance?  If yes, how much does it help?
\end{quote}
It is clear that there are cases when this helps: if, for example, we already know that $|x_i| = 0$ for many indices $i$, then we can simply look for an eigenvector of the induced submatrix of $A$. This is clearly an easier problem because the sub-matrix is going to be smaller. Or, if one entry $|x_i|$ is much larger than all the other entries, then this should also be useful in recovering the eigenvector. The focus of this paper will be on the most general case: we thus make no particular assumptions on $|x_i|_{i=1}^n$, we will only assume w.l.o.g. that they are all nonzero. There is another natural case which we will not address in this paper: if $x \in \mathbb{C}^n$ is complex, then reconstructing $x_i$ from $|x_i|$ would require reconstructing a phase $x_i = e^{i \theta_i} |x_i|$ which is more difficult than recovering a sign $x_i = \varepsilon_i |x_i|$. We will only consider the case of real eigenvalues $\lambda$ and real matrices $A \in \mathbb{R}^{n \times n}$.

We quickly discuss some related problems: a matrix identity dating back to Jacobi \cite{jacobi} and discovered many more times (see Denton, Parke, Tao \& Zhang \cite{denton}) is as follows: using $A_j$ to denote minor of $A$ that results from removing the $j-$th row and column, the $j-$th entry of the $i-$th eigenvector satisfies
$$ |v_{i,j}|^2  \prod_{k=1 \atop k \neq i}^{n} (\lambda_i(A) - \lambda_k(A)) = \prod_{k=1}^{n-1} (\lambda_i(A) - \lambda_k(A_j)).$$
In particular, if we were really good at finding eigenvalues, we could \textit{almost} find eigenvectors at the same speed: the only missing ingredient is recovering the sign.
There are other motivations. One could look at the problem combinatorially and argue that $(|x_i|)_{i=1}^n$ are step sizes on $\mathbb{R}$ and we have to arrange for each step to go either left or right to hit a certain target -- and this has to be done simultaneously for $n$ such problems. This interpretation suggests that the problem should become easier when
 $(|x_i|)_{i=1}^n$ has a few entries that are substantially larger than others. The question is also naturally connected to other problems: the classical phase retrieval problem 
\cite{CLEbook,marchesini, candes2, sun,Mallat2015} asks to recover a vector $x \in \mathbb{R}^n$ (or $\mathbb{C}^n$) from `phaseless' measurements $\left|\left\langle a_i, z\right\rangle\right|$, where $a_i \in \mathbb{R}^n$ (or $\mathbb{C}^n$) for $1 \leq i \leq m$. Our problem is naturally related in the sense that we are trying to solve $\left|\left\langle a_i, x \right\rangle\right| =  | \lambda x_i|$, however, it is clearly a more specialized case since we have additional information about $x$. As such, it may be a useful toy model for the full phase retrieval problem. We conclude with the Komlos conjecture \cite{spe}: given $v_1, \dots, v_n \in \mathbb{R}^n$ such that $\|v_i\| \leq 1$, it is clear that there exists a universal constant $c_n$ (depending only on $n$ but not on the vectors) such that there exist signs $\varepsilon_1, \varepsilon_2, \dots, \varepsilon_n \in \left\{-1,1\right\}$ with
$$ \| \varepsilon_1 v_1 + \varepsilon_2 v_2 + \dots + \varepsilon_n v_n \| \leq c_n.$$
Random signs show that $c_n \leq \sqrt{n}$. The Komlos conjecture asks whether $c_n$ can be taken to be independent of $n$ and there are now algorithmic constructions that lead good choices of signs (see \cite{bana, ban, ban2}). Our result is of a similar flavor except that when considering the eigenvectors $v_1, \dots, v_n \in \mathbb{R}^n$ of the matrix $A - \lambda \cdot \mbox{Id}_{n\times n}$ we know a priori by assumption that there exist signs for which $\varepsilon_1 v_1 + \varepsilon_2 v_2 + \dots + \varepsilon_n v_n = 0$.

\section{ Results}
\subsection{Summary.}  We quickly outline the structure of the paper.
\begin{enumerate}
\item Algorithm 1, presented in \S 2.2, is an iterative, randomized algorithm for recovering the true sign given $A, \lambda$ and $(|x_i|)_{i=1}^n$.
\item One of the nice aspects of Algorithm 1 is that it can be analyzed. We prove that it converges in the limit at a rate depending on the singular values of an associated matrix. This result is stated in \S 2.3.
\item Algorithm 1, in practice, can be slow. \S 2.4 presents Algorithm 2 which is very simple and often faster. Algorithm 2 will not work for small eigenvalues $|\lambda| \sim 0$, however, it is often remarkably effective for $|\lambda|$ large. We explain the underlying reason in \S 4.4.
\item All these points are numerically illustrated in \S 3.
\item The proof of the main result is given in \S 4.
\end{enumerate}

One interesting
aspect of both Algorithm 1 and Algorithm 2 is that they tends to quickly produce estimates for the true sign $(\varepsilon_i)_{i=1}^n$ of which more than $51\%$ are correct. This suggests a
secondary question (which we do not pursue in this paper): is it significantly easier to recover a nontrivial fraction of the true signs (say, $51\%$) than it is to recover all of them? One would perhaps assume that recovering $51\%$ is the difficult part. It would be interesting to make this notion precise.

\subsection{Algorithm 1.}
This section introduces Algorithm 1. Note that whenever $|x_i| = 0$, then $x_i = 0$ and we can remove these coordinates from consideration. This allows us to assume without loss of generality that $x_i \neq 0$ for all $1 \leq i \leq n$.
Our initial reduction will be to make the ansatz $x_i = \varepsilon_i |x_i|$. 
Then, for all $1 \leq i \leq n$,
$$ \sum_{j=1}^{n} a_{ij} \varepsilon_j |x_j| = \lambda \varepsilon_i |x_i|$$
and thus, for all $1\leq i \leq n$,
$$ \sum_{j=1}^{n} a_{ij}  \frac{|x_j|}{|x_i|} \varepsilon_j = \lambda \varepsilon_i$$
This means that the sign vector $\varepsilon \in \left\{-1,1\right\}^n$ is an eigenvector of eigenvalue $\lambda$ of the matrix $B \in \mathbb{R}^{n \times n}$, where
$$ B_{ij} = a_{ij} \frac{|x_j|}{|x_i|}.$$
That is, $B=D^{-1}AD$, where $D\in \mathbb{R}^{n\times n}$ is a diagonal matrix with $D_{ii}=|x_i|$.
The true sign vector $\varepsilon \in \left\{-1,1\right\}^n$ is a non-trivial solution of the linear system
$$ (B- \lambda \cdot \mbox{Id}_{n \times n}) y = 0$$
and we will assume that this solution is unique.
We introduce the matrix 
$$C = B- \lambda \cdot \mbox{Id}_{n \times n} \in \mathbb{R}^{n \times n}$$ and use $c_1, \dots, c_n \in \mathbb{R}^{n \times n}$ to denote its $n$ rows. One way of interpreting $C$ is that it normalizes the matrix $B$ so that each sign $x_i = \varepsilon_i |x_i|$ has the same amount of weight independently of the size of $|x_i|$.
\begin{quote}
\textbf{Algorithm 1.} Start with an arbitrary initial guess $0 \neq y_0 \in \mathbb{R}^n$. 
\begin{enumerate}
\item Given $y_k$, generate a random number $i \in \left\{1,2,\dots,n\right\}$ where the likelihood of choosing $1 \leq j \leq n$ is given by
$$ \mathbb{P}(i = j) = \frac{\|c_j\|^2}{\|C\|_F^2}.$$
\item Set
$$ y_{k+1} = y_k - \left\langle y_k, \frac{c_i}{\|c_i\|} \right\rangle \frac{c_i}{\|c_i\|}.$$
\item When stopping after $m$ steps, use the signs of $y_m$, $(\mbox{sign}((y_m)_i))_{i=1}^n$, as an approximation for the true signs $\varepsilon$. 
\end{enumerate}
\end{quote}
We quickly discuss what to expect: it is intuitive to suspect that the problem will become more difficult if there are other eigenvectors of $A$ with a similar eigenvalue and eigenvalues that have entries with comparable size. 
 It is certainly conceivable that a matrix $A$ has, for two eigenvalues $\lambda_1 \neq \lambda_2$ two eigenvectors $Ax_1 = \lambda_1 x_1$ and $A_2 x_2 = \lambda_2 x_2$ such that, for each $1 \leq i \leq n$, we have $|(x_1)_i| = |(x_2)_i|$ for all $1 \leq i \leq n$.  So having accurate information on the eigenvalue is certainly required, since $\lambda_1$ and $\lambda_2 $ could be arbitrarily close to each other: choosing a small perturbation of the eigenvalue may lead to a different solution.
This has to be incorporated in the statement, and indeed our analysis will be phrased in terms of the spectral gap of $C$. We denote the singular values of a matrix $M\in \mathbb{R}^{n\times n}$ by $$\sigma_1(M) \geq \sigma_2(M) \geq \dots \geq \sigma_n(M).$$ Since $C$ has a nontrivial null space, $C \varepsilon = 0$, we have $\sigma_n(C) = 0$. We will assume that $\sigma_{n-1}(C) > 0$ to recover some form of uniqueness.  The other relevant quantity in our analysis is the Frobenius norm $\|C\|_F^2 = \mbox{tr}~(C C^T)$.
We can now informally state properties of the algorithm.
\begin{enumerate}
\item In a suitable sense,  $(y_m) \rightarrow \lambda \varepsilon$, where $\lambda \in \mathbb{R}$, in expectation as $m \rightarrow \infty$. The algorithm recovers a multiple of the ground truth in expectation.
\item We expect, in a sense to be made precise, that $(\mbox{sign}(y_m)_i)_{i=1}^n$ coincides with $(\varepsilon_i)_{i=1}^n$ with some nontrivial likelihood as soon as
$$ m \gtrsim   \frac{ \|C\|_F^2}{\sigma_{n-1}^2(C)} \log{n}.$$
\item It is to be expected that $(\mbox{sign}((y_m)_i))_{i=1}^n$ and $\varepsilon$ coincide in a nontrivial fraction
of entries (exceeding $50\%$) already for smaller values of $m$. 
\end{enumerate}

We note that Algorithm 1 can be interpreted as a simplified version of the Random Kaczmarz method proposed by Strohmer \& Vershynin \cite{strohmer} applied to the problem of finding the nullspace of a matrix (see Steinerberger \cite{stein}): a crucial new ingredient is that we know the solution to be an element in $\left\{-1,1\right\}^n$ which allows for an effective rounding procedures. We also refer to work of Jeong \& G\"unt\"urk \cite{jeong} and Tan \& Vershynin \cite{shuo} (see also Wei \cite{wei}) who used Random Kaczmarz for the classical phase retrieval problem which is somewhat related to our approach.

\subsection{The Theorem: Analysis of Algorithm 1.} This section describes some properties of Algorithm 1. We will prove an inequality of the following type: the number of signs that are incorrectly recovered by $y_k$ multiplied with the square of $\|y_k\|$ is, in expectation, an exponentially decaying quantity in $k$: either the number of incorrectly predicted
signs or the norm $\|y_k\|$ has to decay in $k$. This is complemented by a second bound stating that if $y_0$ is chosen uniformly at random from a sphere around the origin, then $\mathbb{E} \|y_k\|$ is bounded from below uniformly in $k$:  the exponential decay is given by the number of incorrectly predicted signs.
There is a small ambiguity: if $Ax = \lambda x$ is an eigenvector, then $-x$ is an element from the same eigenspace and can be viewed as the `same' eigenvector. Therefore, the `correct' signs $\varepsilon \in \left\{-1,1\right\}^n$ as defined via the equation $x_i = \varepsilon_i |x_i|$ are really only defined up to sign in the global sense: if $\varepsilon$ is a correct choice of signs, then so is $-\varepsilon$. We fix one of the two admissible choices, call it $\varepsilon$ and introduce the set of incorrectly recovered signs
$$ S =   \left\{1 \leq i \leq n: \mbox{sign}((y_k)_i) \neq \varepsilon_i \right\}.$$
The goal is to have the size of $S$ to be either as small as possible (if $\#S =0$, then the signs of $y_k$ and $\varepsilon$ coincide) or as large as possible (if $\#S = n$, then the signs of $y_k$ and $-\varepsilon$ coincide which would then correctly reproduce the eigenvector $-x$).
\begin{thm}[Main Result] If $\sigma_{n-1}(C) > 0$, then
  $$ \mathbb{E}\left[ \min\left\{\#S, n- \#S \right\} \cdot \|y_k\|^2 \right] \leq  n  \left(1 - \frac{\sigma_{n-1}^2(C)}{ \|C\|_F^2} \right)^k \cdot \left\| y_0 \right\|^2.$$
If $y_0$ is chosen uniformly at random from a sphere centered at the origin, then 
$$ \forall~k \in \mathbb{N} \qquad \mathbb{E}~ \|y_k\|^2 \geq \frac{\|y_0\|^2}{n}.$$ 
\end{thm}

This result shows that we can expect complete recovery for a range of
$$ k \gtrsim \frac{ \|C\|_F^2}{\sigma_{n-1}^2(C)} \log{n}.$$
This may at first glance be too good to be true since the dimension only appears logarithmically. However, recalling that
$$ \sigma_{n-1}^2(C) \leq \frac{1}{n-1} \sum_{k=1}^{n-1} \sigma_k(C)^2 \leq \frac{1}{n-1} \|C\|_F^2,$$
we see that we require, independently of $C$, at least $k \gtrsim n \log{n}$ iteration steps.

\subsection{Algorithm 2.} We now present Algorithm 2. We follow the same type of reduction as in Algorithm 1 and arrive at the problem of finding
a vector $(\varepsilon_i)_{i=1}^{n} \in \left\{-1,1\right\}^n$ in the nullspace of a given matrix
$$C = B- \lambda \cdot \mbox{Id}_{n \times n} \in \mathbb{R}^{n \times n}.$$ 
\begin{quote}
\textbf{Algorithm 2.} Start with an arbitrary initial guess $ \varepsilon_0 \in \left\{-1,1\right\}^n$. 
\begin{enumerate}
\item Given $\varepsilon_k$, compute the vector $v = C \varepsilon_k$.
\item Generate a random number $i \in \left\{1,2,\dots,n\right\}$ where the likelihood of choosing $1 \leq j \leq n$ is given by
$$ \mathbb{P}(i = j) = \frac{|v_j|^2}{\|v\|_{}^2}.$$
\item Define $\varepsilon_{k+1}$ by taking $\varepsilon_k$ and flipping its $i'$th entry, i.e.
$$ (\varepsilon_{k+1})_{i} = -(\varepsilon_k)_i$$
while keeping all other entries the same.
\end{enumerate}
\end{quote}

\section{Numerical Examples}
The purpose of this section is to give numerical examples that illustrate the performance of the algorithms on highly structured and highly random matrices. 

\subsection{Hadamard matrices.} We start with an explicit example: pick $A$ to be a $256 \times 256$ Hadamard matrix
and then set $A_{11} = 2$. $A$ has one eigenvalue $\lambda_1 \sim 2$ and one eigenvalue $\lambda_{256} \sim -0.5$, all other eigenvalues are $\pm 1$. We try to
recover the eigenvector corresponding to the largest eigenvalue, see Fig. \ref{fig:1}.

\begin{figure}[h!]
\begin{minipage}[l]{.48\textwidth}
\begin{tikzpicture}
\node at (0,0) {\includegraphics[width = \textwidth]{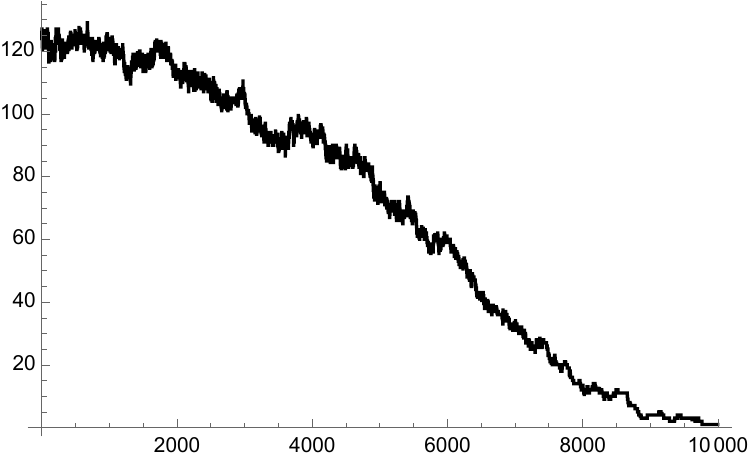}};
\end{tikzpicture}
\end{minipage} 
\begin{minipage}[r]{.48\textwidth}
\begin{tikzpicture}
\node at (0,0) {\includegraphics[width = \textwidth]{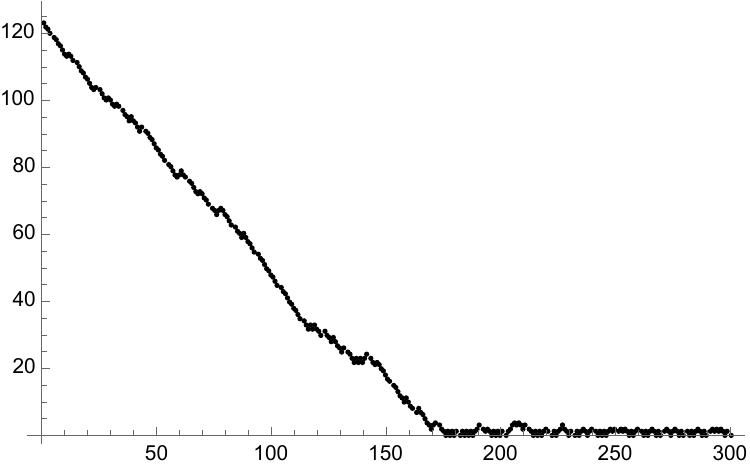}};
\end{tikzpicture}
\end{minipage} 
\caption{Number of incorrect signs ($y-$axis) vs. iteration steps ($x-$axis). Left: sample run of Algorithm 1 for $\lambda_1$. It correctly recovers all signs within a couple of thousand iteration steps. Right: Algorithm 2 recovers the ground truth almost immediately.} 
\label{fig:1}
\end{figure}
\begin{figure}[h!]
\begin{minipage}[l]{.48\textwidth}
\begin{tikzpicture}
\node at (0,0) {\includegraphics[width = \textwidth]{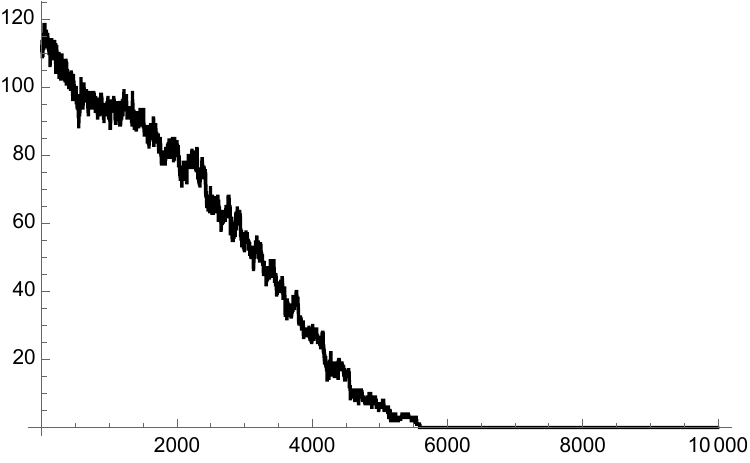}};
\end{tikzpicture}
\end{minipage} 
\begin{minipage}[r]{.48\textwidth}
\begin{tikzpicture}
\node at (0,0) {\includegraphics[width = \textwidth]{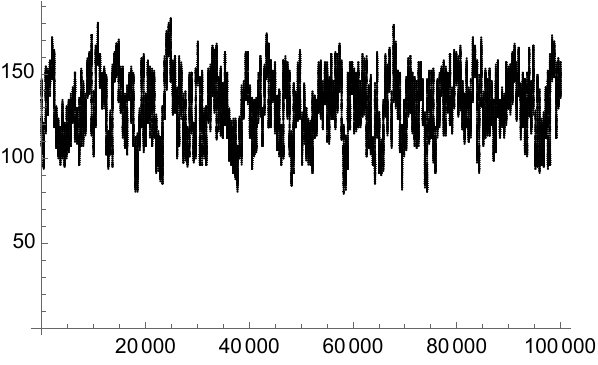}};
\end{tikzpicture}
\end{minipage} 
\caption{Number of incorrect signs ($y-$axis) vs. iteration steps ($x-$axis). Left: sample run of Algorithm 1 for $\lambda_{256}$. It correctly recovers the signs within a couple of thousand iteration steps. Right: Algorithm 2 does not seem to converge at all.}
\label{fig:2} 
\end{figure}

We will now switch gears, see Fig. \ref{fig:2}, and try to instead recover the eigenvector corresponding to the smallest (absolute) eigenvalue $\lambda_{256} \sim -0.5$. Algorithm 1 behaves in
an essentially the same way: it trends towards recovering the ground truth and achieves perfect recovery typically within a couple of thousand iterations at roughly the same rate as the eigenvector
associated to the largest eigenvalue. In contrast, Algorithm 2 does not seem to converge at all, even when run for many more iteration steps. This is explained in \S 4.4 and illustrates the point that Algorithm 2 can only be effective for eigenvalues for which $|\lambda|$ is disproportionately large compared to the other eigenvalues.

\subsection{A Random Matrix.}  Our second example (see Fig. 3) will be as follows: we choose a matrix $A \in \mathbb{R}^{100 \times 100}$ to have i.i.d. Gaussian entries $a_{ij} \sim \mathcal{N}(0,1)$. We then replace $A_{11}$ by 50 to create a nice spectral gap. Algorithm 1 works and has a nice flow evolving towards the ground truth. Since the eigenvalue is extremal, we expect Algorithm 2 to perform much better. Indeed, it recovers the ground truth within relatively few iterations.

\begin{figure}[h!]
\begin{minipage}[l]{.48\textwidth}
\begin{tikzpicture}
\node at (0,0) {\includegraphics[width = \textwidth]{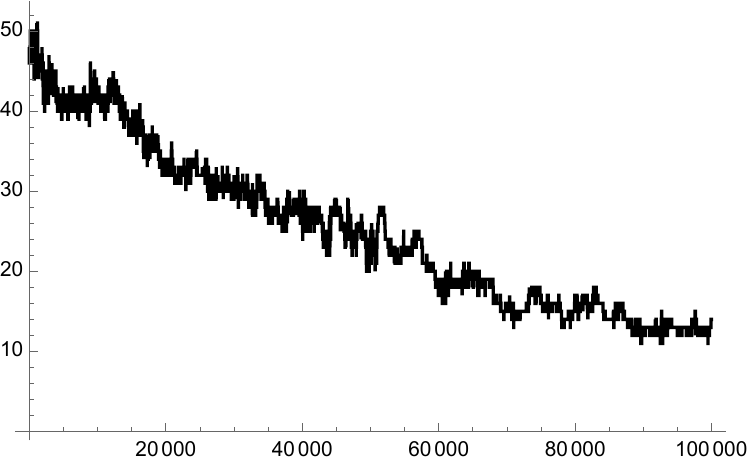}};
\end{tikzpicture}
\end{minipage} 
\begin{minipage}[r]{.48\textwidth}
\begin{tikzpicture}
\node at (0,0) {\includegraphics[width = \textwidth]{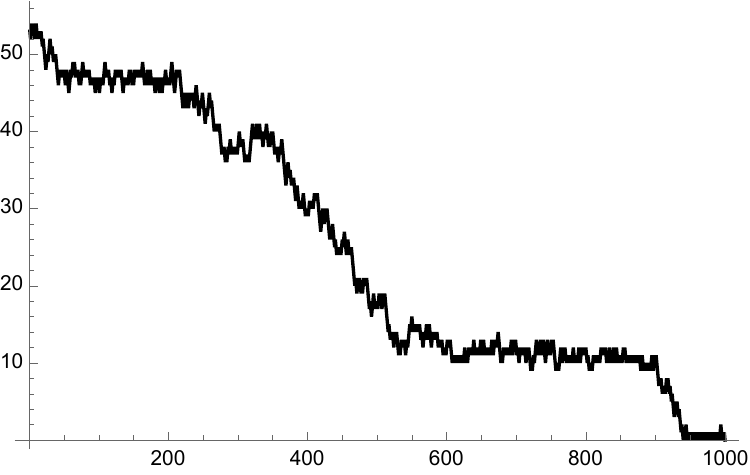}};
\end{tikzpicture}
\end{minipage} 
\caption{Left: a sample run of Algorithm 1 on a random Gaussian matrix. Right: Algorithm 2 on the same matrix.} 
\end{figure}

\section{Proof}
 We start with an outline of the structure of the argument.
We use $\varepsilon \in \left\{-1,1\right\}^n$ to denote the true solution and will analyze the algorithm with respect to the orthogonal decomposition
$$ y_k = \left\langle y_k, \frac{\varepsilon}{\|\varepsilon\|}\right\rangle \frac{\varepsilon}{\|\varepsilon\|} + \left(y_k -  \left\langle y_k, \frac{\varepsilon}{\|\varepsilon\|}\right\rangle \frac{\varepsilon}{\|\varepsilon\|}\right).$$
Note that the first term is merely the orthogonal projection onto the one-dimensional subspace spanned by $\varepsilon$, the second term is everything else.
We will show (Lemma 1) that the first term fluctuates a little in $k$ but preserves its value in expectation. The second term, on the other hand, decays exponentially
in expectation. We will then argue (Lemma 2) that the size of the second term yields an upper bound on the number of entries of $y_k$ whose sign disagrees with
the corresponding sign of $\varepsilon$. These two estimates then imply the Theorem.

\subsection{ Lemma 1: exponential shrinking.}
\begin{lemma} For all $k \in \mathbb{N}$,
$$ \mathbb{E} \left\langle y_k, \varepsilon \right\rangle = \left\langle y_0, \varepsilon \right\rangle.$$
The orthogonal projection of $y_k$ in directions orthogonal to $\varepsilon$ satisfies
$$ \mathbb{E}~\left\|y_k - \left\langle y_k, \frac{\varepsilon}{\|\varepsilon\|} \right\rangle \frac{\varepsilon}{\|\varepsilon\|} \right\|^2 \leq \left(1 - \frac{\sigma_{n-1}^2(C)}{ \|C\|_F^2} \right)^k \cdot \left\|y_0 - \left\langle y_0, \frac{\varepsilon}{\|\varepsilon\|} \right\rangle \frac{\varepsilon}{\|\varepsilon\|} \right\|^2.$$
\end{lemma}

\begin{proof} 
We write $y_k$ as the sum of the projection onto $\varepsilon$ and the rest. Therefore, for each $k$, the decomposition $y_k = \pi_k + r_k$, where 
  $$ \pi_k = \left\langle y_k, \frac{\varepsilon}{\|\varepsilon\|} \right\rangle \frac{\varepsilon}{\|\varepsilon\|}$$
  and 
  $$ r_k = y_k - \pi_k.$$
  The Pythagorean theorem tells us that 
  $$ \|y_k\|^2 = \|\pi_k\|^2 + \|r_k\|^2.$$
  Our desired inequality can be phrased as an inequality for $\|r_k\|^2$.
We now observe a deterministic identity: since $C \varepsilon = 0$, we have $\left\langle c_i, \varepsilon\right\rangle = 0$ for all $1 \leq i \leq n$ and thus, independently of which $1 \leq j \leq n$ is chosen, that
\begin{align*}
 y_{k+1} &= y_k - \left\langle y_k, \frac{c_j}{\|c_j\|} \right\rangle \frac{c_j}{\|c_j\|} \\
 &= \pi_k + r_k -  \left\langle \pi_k + r_k, \frac{c_j}{\|c_j\|} \right\rangle \frac{c_j}{\|c_j\|} \\
 &= \pi_k + r_k - \left\langle  r_k, \frac{c_j}{\|c_j\|} \right\rangle \frac{c_j}{\|c_j\|}.
 \end{align*}
From this, we infer that $\pi_{k+1} = \pi_k$ whereas $r_k$ undergoes a random evolution and
$$ r_{k+1} =  r_k - \left\langle  r_k, \frac{c_j}{\|c_j\|} \right\rangle \frac{c_j}{\|c_j\|} \qquad \mbox{with likelihood} \quad \frac{\|c_j\|^2}{\|C\|_F^2}.$$
We can now compute, conditional on $r_k$, that
\begin{align*}
 \mathbb{E}~  \|r_{k+1}\|^2 &=  \sum_{j=1}^{n} \frac{\|c_j\|^2}{\|C\|_F^2}   \left\| r_k - \left\langle r_k, \frac{c_j}{\|c_j\|} \right\rangle \frac{c_j}{\|c_j\|} \right\|^2 \\
 &= \|r_k\|^2 -  \frac{1}{\|C\|_F^2} \sum_{j=1}^{n} \left\langle r_k, c_j \right\rangle^2 \\
 &=  \|r_k\|^2 -  \frac{\| C r_k\|^2}{\|C\|_F^2}.
 \end{align*}
Since $r_k$ is orthogonal to $\varepsilon$ (which is the unique vector associated to the smallest singular value), we have $\| C r_k\|^2 \geq \sigma_{n-1}^2(C) \|r_k\|^2$ and therefore
$$ \mathbb{E}~ \|r_{k+1}\|^2 \leq \left(1 - \frac{\sigma_{n-1}^2(C)}{\|C\|_F^2}\right) \|r_k\|^2.$$
Iterating the inequality gives the desired result.
\end{proof}

\subsection{A simple inequality.}
We will also use an elementary inequality for vectors to show that the second term in our orthogonal
decomposition is directly connected to the number of incorrectly recovered signs. The inequality is independent of the algorithm and true for all vectors $y \in \mathbb{R}^n$.
\begin{lemma} For any $\varepsilon \in \left\{-1,1\right\}^n$ and any $y \in \mathbb{R}^n$, denote the number of coordinates where $y$ and $\varepsilon$ have a different sign by
$$ S = \left\{1 \leq i \leq n: \emph{sign}(y_i) \neq \varepsilon_i \right\}.$$
Then
$$  \left\| y -  \left\langle y, \frac{\varepsilon}{\|\varepsilon\|}\right\rangle \frac{\varepsilon}{\|\varepsilon\|} \right\|^2 \geq \frac{ \min\left\{ \# S, n- \# S \right\}}{n} \cdot \|y\|^2.$$
\end{lemma}
\begin{proof}
We start with the orthogonal decomposition
$$ y = \left\langle y, \frac{\varepsilon}{\|\varepsilon\|}\right\rangle \frac{\varepsilon}{\|\varepsilon\|} + \left(y -  \left\langle y, \frac{\varepsilon}{\|\varepsilon\|}\right\rangle \frac{\varepsilon}{\|\varepsilon\|}\right).$$
The Pythagorean theorem implies that
$$ \|y\|^2 =  \left\langle y, \frac{\varepsilon}{\|\varepsilon\|}\right\rangle^2 + \left\| y -  \left\langle y, \frac{\varepsilon}{\|\varepsilon\|}\right\rangle \frac{\varepsilon}{\|\varepsilon\|} \right\|^2.$$
We now 
observe
\begin{align*}
\left\langle y, \frac{\varepsilon}{\|\varepsilon\|}\right\rangle^2 &= \frac{1}{n} \left\langle y, \varepsilon\right\rangle^2 = \frac{1}{n} \left( \sum_{i=1}^{n} y_i \varepsilon_i \right)^2 = \frac{1}{n} \left( \sum_{ i \in S}^{} y_i \varepsilon_i  +  \sum_{ i \notin S}^{} y_i \varepsilon_i \right)^2 \\
&=  \frac{1}{n} \left( \sum_{i \in S}^{} -|y_i| +  \sum_{ i \notin S}^{} |y_i | \right)^2 = \frac{1}{n} \left( \sum_{i =1}^{n} |y_i| \left(1 - 2 \cdot 1_{i \in S}\right) \right)^2.
\end{align*}
This sum is being squared, therefore we have to bound it from above and below. Bounding it from above is easy: certainly the sum becomes larger if $y$ vanishes in all coordinates indexed by the set $S$ and the Cauchy-Schwarz inequality implies
$$ \sum_{i =1}^{n} |y_i| \left(1 - 2 \cdot 1_{i \in S}\right) \leq \sum_{i  \in S^c}^{} |y_i| \leq  \sqrt{\# S^c} \cdot  \|y\|.$$
Minimizing the sum is similar, the sum becomes smaller if $y$ vanishes in all the coordinates indexed by $S^c$ and, again via Cauchy-Schwarz,
$$  \sum_{i =1}^{n} |y_i| \left(1 - 2 \cdot 1_{i \in S}\right) \geq -\sum_{i  \in S}^{} |y_i| \geq - \sqrt{\# S} \cdot  \|y\|.$$
Therefore,
$$  \frac{1}{n} \left( \sum_{i =1}^{n} |y_i| \left(1 - 2 \cdot 1_{i \in S}\right) \right)^2 \leq \frac{\max\left\{\# S, \# S^c \right\}}{n} \|y\|^2$$
and thus, since $\# S + \# S^c = n$,
$$  \left\| y -  \left\langle y, \frac{\varepsilon}{\|\varepsilon\|}\right\rangle \frac{\varepsilon}{\|\varepsilon\|} \right\|^2 \geq \frac{\min\left\{\# S, n - \# S \right\}}{n} \|y\|^2.$$
\end{proof}

\subsection{Proof of the Theorem}
\begin{proof}

 Lemma 2 implies that for the set 
 $$S=   \left\{1 \leq i \leq n: \mbox{sign}((y_k)_i) \neq \varepsilon_i \right\},$$
  we have
  $$   \frac{ \min\left\{ \# S, n- \# S \right\}}{n} \cdot \|y_k\|^2 \leq  \left\| y_k -  \left\langle y_k, \frac{\varepsilon}{\|\varepsilon\|}\right\rangle \frac{\varepsilon}{\|\varepsilon\|} \right\|^2$$
Taking an expectation over both sides and using Lemma 1 leads to
\begin{align*}
\mathbb{E}  \frac{ \min\left\{ \# S, n- \# S \right\}}{n} \cdot \|y_k\|^2 &\leq \mathbb{E} \left\| y_k -  \left\langle y_k, \frac{\varepsilon}{\|\varepsilon\|}\right\rangle \frac{\varepsilon}{\|\varepsilon\|} \right\|^2 \\
&\leq  \left(1 - \frac{\sigma_{n-1}^2(C)}{ \|C\|_F^2} \right)^k \cdot \left\|y_0 - \left\langle y_0, \frac{\varepsilon}{\|\varepsilon\|} \right\rangle \frac{\varepsilon}{\|\varepsilon\|} \right\|^2 \\
&\leq  \left(1 - \frac{\sigma_{n-1}^2(C)}{ \|C\|_F^2} \right)^k \cdot \left\|y_0 \right\|^2.
\end{align*}

It remains to prove the lower bound on $\|y_k\|$. Naturally, for any $y_k \in \mathbb{R}^n$, we have (using the proof of Lemma 1)
$$ \|y_k\|^2 \geq  \left\langle y_k, \frac{\varepsilon}{\|\varepsilon\|} \right\rangle^2 =  \frac{1}{n}  \left\langle y_k, \varepsilon \right\rangle^2 = \frac{1}{n} \left\langle y_0, \varepsilon\right\rangle^2.$$
 \end{proof}

\subsection{A Look at Algorithm 2.} We conclude by discussing Algorithm 2. Our observations will be heuristic and apply to generic matrices but will not be rigorous at the level of being applicable to all matrices. Algorithm 2 proceeds by analyzing
$$ C \varepsilon_k = B \varepsilon_k - \lambda \varepsilon_k$$
and hoping that the size of the entries are a reasonable indicator of `how incorrect' the corresponding entry is. Larger entries are more likely to be randomly chosen and flipped.
Let us consider the size of the $i-$th entry: its size is given by 
$$(C \varepsilon)_i =  \sum_{j=1}^{n} a_{ij}  \frac{|x_j|}{|x_i|} \varepsilon_j - \lambda \varepsilon_i$$
If we now assume most entries to be roughly correct then $(C \varepsilon)_i$ will be roughly $\sim 0$ if $\varepsilon_i$ has the correct sign and roughly $\sim  \pm 2\lambda$ if the sign $\varepsilon_i$ was chosen incorrectly. This suggests that flipping $\varepsilon_i$ with likelihood proportional to $(C \varepsilon)_i$ might be a good idea when $\lambda$ is large: things will surely get harder when $\lambda$ is small.
We conclude by estimating the scaling of the relevant quantities when $A$ is a random matrix. Let us assume, for simplicity, that the entries of the eigenvector $|x_i| \sim |x_j|$ are roughly comparable (up to some constant factors). Then, for randomly chosen signs $\varepsilon \in \left\{-1,1\right\}^n$, we expect
$$  \sum_{j=1}^{n} a_{ij}  \frac{|x_j|}{|x_i|} \varepsilon_j  \quad \mbox{ to follow roughly }  \quad \mathcal{N}\left(0, \sigma_i^2\right),$$
where the variance is given by
$$ \sigma_i^2 =  \sum_{j=1}^{n} a_{ij} ^2 \frac{|x_j|^2}{|x_i|^2} \sim  \sum_{j=1}^{n} a_{ij}^2.$$
Assuming that $A$ is a random matrix, all the rows have comparable norm and we expect that
$$ \sigma_i^2  \sim  \sum_{j=1}^{n} a_{ij}^2 \sim \frac{1}{n} \sum_{i,j=1}^{n} a_{ij}^2 = \frac{1}{n} \|A\|_F^2 = \frac{1}{n} \sum_{i=1}^{n} \sigma_i(A)^2 \leq \|A\|^2,$$
where the last inequality is typically far from sharp (and only sharp if all singular values are identical). This means that as soon as $\lambda$ is an eigenvalue close to the operator norm $|\lambda| \sim \| A\|$, it will typically dominate expressions of the type
$$(C \varepsilon)_i =  \sum_{j=1}^{n} a_{ij}  \frac{|x_j|}{|x_i|} \varepsilon_j - \lambda \varepsilon_i$$
and Algorithm 2 will be effective in recovering the sign of the eigenvector (see \S 3).

\end{document}